\documentclass[11pt]{article}
\usepackage{verbatim}
\usepackage{amsmath}
\usepackage{amssymb}
\usepackage{amsfonts}
\usepackage{amsthm}
\usepackage{cite}
\usepackage{graphicx} %,obrázky
\usepackage{pgf,tikz}
\usepackage{tikz-cd} % diagramy
\usepackage{pgfplots}
\usepgfplotslibrary{fillbetween}
\pgfplotsset{compat = newest}
\usepackage{mathrsfs}
\usetikzlibrary{arrows,intersections}
\usepackage{color}
\usepackage[legalpaper, margin=1in]{geometry}
\usepackage{enumitem}
\usepackage{latexsym}
\usepackage{caption}

\newtheorem{theorem}{Theorem}[section]
\newtheorem*{theorem*}{Theorem}
\newtheorem{corollary}[theorem]{Corollary}
\newtheorem{lemma}[theorem]{Lemma}
\newtheorem{proposition}[theorem]{Proposition}

\theoremstyle{definition}

\newtheorem{claim}{Claim}[theorem]
\theoremstyle{remark}
\newtheorem*{claimproof}{Proof}

\newcommand{\R}{\mathbb{R}}
\newcommand{\N}{\mathbb{N}}

\newcommand{\claimend}{{\hfill $\blacksquare$}}
\renewcommand{\restriction}{\mathord{\upharpoonright}}

\title{Embedding $C(K)$ for countable $K$ into $C([0,1])$ as Besicovitch functions}
\author{
Jan Dud\'ak\footnote{https://orcid.org/0000-0003-0627-6641}\\
Department of Mathematics\\
Faculty of Civil Engineering, Czech Technical University\\
Prague, Czechia\\
E-mail: jan.dudak.2@cvut.cz}

\begin{document}
\maketitle

\renewcommand{\thefootnote}{}



\renewcommand{\thefootnote}{\arabic{footnote}}
\setcounter{footnote}{0}

\begin{abstract}
Extending a recent result from [Bull. Belg. Math. Soc. Simon Stevin 33 (2026), 138–144], we prove that the space of continuous functions $C(X)$ on any countable compact space $X$ admits an isometric copy in $C([0,1])$ consisting, except for the zero function, entirely of Besicovitch functions, i.e., functions that have no one-sided derivative (finite or infinite) at any point.
\end{abstract}

\section{Introduction}

Over the past few decades, identifying linear structures of mathematical objects with specific properties has emerged as a distinct trend across diverse subfields of mathematics. This widespread interest has generated an extensive body of literature spanning linear chaos, real and complex analysis, set theory, linear and multilinear algebra, operator theory, topology, measure theory, functional analysis, and abstract algebra. (see, e.g., \cite{L01,L02,L03,L04,L05}).

Given a property of functions, we say that the subset $M$ of $C([0,1])$ consisting of functions which have the property is {\em lineable} if $M \cup \{0\}$ contains an infinite-dimensional subspace. The set $M$ will be called {\em spaceable} if $M \cup \{0\}$ contains a closed infinite-dimensional subspace.

Let us provide a brief overview of lineability and spaceability results related to the differentiability of functions. It is clear that the set of everywhere differentiable functions on $[0,1]$ is a linear subspace of $C([0,1])$ and hence it is lineable. However, as shown by V. I. Gurariy \cite{G-1966}, the set of everywhere differentiable functions on $[0,1]$ is not spaceable, i.e. there is no infinite-dimensional closed subspace of $C([0,1])$ in which every function is differentiable everywhere on $[0,1]$. On the other hand, it was shown in the same article that the set of {\em nowhere} differentiable continuous functions on $[0,1]$ is lineable. This was improved by Fonf, Gurariy and Kade\v{c} \cite{FGK}, with a preprint appearing as early as 1990, who showed that the set of nowhere differentiable continuous functions on $[0,1]$ is spaceable; that is, there is a closed, infinite-dimensional subspace of $C([0,1])$ every non-zero element of which is nowhere differentiable on $[0,1]$. In fact, much more is true; L. Rodr\'iguez-Piazza \cite{LRP} proved that there exists a closed subspace of $C([0,1])$ which is universal for separable Banach spaces (i.e. every separable Banach space can be linearly isometrically embedded into it) such that every nonzero function in the subspace is nowhere differentiable. An even stronger result was obtained by S. Hencl \cite{H} who showed that there exists a closed subspace of $C([0,1])$ which is universal for separable Banach spaces and in which every nonzero function is nowhere approximately differentiable and nowhere H\"older.

Given a real-valued function $f$ of one real variable, we say that $f$ does not have a one-sided derivative at a point $x$ if neither the left nor the right derivative of $f$ at $x$ exists, not even as an infinite value. A continuous function having no one-sided derivative at any point is called a Besicovitch function. It was shown in \cite{JB15,JB26} that the set of Besicovitch functions on $[0,1]$ is spaceable. This result was recently improved in \cite{BobokDudak2026}, where it was shown that there is a linear isometric embedding of the space $\boldsymbol{c}$ of convergent sequences (with the supremum norm) into $C([0,1])$ such that every nonzero function in the image is Besicovitch. Since $\boldsymbol{c}$ is linearly isometric to $C(K)$, where $K$ is the one-point compactification of $\N$, a natural question arises: for which compact metrizable spaces $X$ does $C([0,1])$ contain a linearly isometric copy of $C(X)$ consisting entirely of Besicovitch functions (up to the zero function)?

In this paper we provide a complete answer to the question for the class of countable compact spaces.

\begin{theorem}\label{MainResult}
    For any countable compact space $X$, there exists a subspace $Y$ of $C([0,1])$ linearly isometric to $C(X)$ such that every nonzero $f \in Y$ is Besicovitch.
\end{theorem}

Regarding the uncountable case, let us note that since $C(X)$ is universal for separable Banach spaces whenever $X$ is an uncountable compact metrizable space, a positive answer for any such space would automatically solve the problem for all of them, echoing the result of Rodr\'iguez-Piazza. Specifically, it would imply the existence of a closed subspace of $C([0,1])$ that is universal for separable Banach spaces and in which every nonzero function is Besicovitch. Whether such a subspace exists remains an open question.

It is worth pointing out that, while the set of all nowhere differentiable functions on $[0,1]$ is comeager in $C([0,1])$, the set of all Besicovitch functions on $[0,1]$ is meager in $C([0,1])$ \cite{Saks1932}. This Baire category distinction, however, implies nothing regarding this open question.

For further details, we refer the interested reader to the monograph by Jarnicki and Pflug \cite{Monsters}, where various types of continuous, nowhere differentiable functions are discussed and treated in depth.

\section{Preliminaries}\label{SectionPreliminaries}

We denote the set of positive integers by $\N$. Throughout this paper, all topological spaces under consideration are assumed to be Hausdorff; in particular, a compact space always means a compact Hausdorff space. Under this convention, every countable compact space is automatically metrizable and, in fact, homeomorphic to a compact subset of $\R$.

For a continuous, real-valued function $f$ on a compact space $X$, we denote its supremum norm by $\lVert f \rVert$. By a nonzero function, we mean a function that is not identically zero.

For any subset $A$ of a topological space we denote by $A'$ the derived set of $A$. Using transfinite recursion, we define the Cantor-Bendixson derivative of order $\alpha$ of $A$, denoted by $A^{(\alpha)}$, for any ordinal number $\alpha$ as follows:
\begin{equation*}
\displaystyle
A^{(\alpha)}=
    \begin{cases}
        A &\textup{if }\alpha=0,\\
        \big(A^{(\beta)}\big)' &\textup{if }\alpha=\beta+1,\\
        \displaystyle \bigcap_{\beta<\alpha} A^{(\beta)} &\textup{if }\alpha \textup{ is a limit ordinal}.
    \end{cases}
\end{equation*}
For a nonempty countable compact space $X$ there is a unique ordinal $\alpha < \omega_1$ for which $X^{(\alpha)}$ is a nonempty finite set. We call this ordinal the Cantor-Bendixson rank of $X$ and denote it by $\textup{CB}(X)$. Equivalently, $\textup{CB}(X)$ is the least $\alpha < \omega_1$ such that $X^{(\alpha+1)}=\emptyset$. It is worth noting that in most publications, the Cantor-Bendixson rank of a topological space $X$ is defined as the least ordinal $\alpha$ satisfying $X^{(\alpha)}=X^{(\alpha +1)}$, which, in the case of countable compact spaces, simplifies to the least $\alpha$ with $X^{(\alpha)}= \emptyset$. 

For any $\alpha \in \omega_1$ and $k \in \omega \setminus \{ 0 \}$, denote by $\theta(\alpha,k)$ the ordinal number given by
\begin{equation*}
    \theta(\alpha,k)=
    \begin{cases}
        k &\textup{if } \alpha=0,\\
        \omega^\alpha \cdot k+1 &\textup{if } \alpha \in \omega_1 \setminus \{ 0 \}.
    \end{cases}
\end{equation*}
Equipped with the order topology, $\theta(\alpha,k)$ forms a nonempty countable compact space. Conversely, as shown by Mazurkiewicz and Sierpi\'nski in 1920 \cite{MazurkiewiczSierpinski}, if $X$ is a nonempty countable compact space, then $X$ is homeomorphic to $\theta(\alpha,k)$, where $\alpha =\textup{CB}(X) \in \omega_1$ and $k \in \omega \setminus \{ 0 \}$ is the cardinality of $X^{(\alpha)}$.

Note that for any $\alpha \in \omega_1 \setminus \{ 0 \}$, we have $(\omega^\alpha+1)^{(\alpha)}=\{ \omega^\alpha \}$. Consequently, if $K$ is a nonempty compact subset of $\omega^\alpha+1$ not containing the point $\omega^\alpha$, then
\[ K^{(\alpha)}\subseteq K \cap (\omega^\alpha+1)^{(\alpha)} = K \cap \{ \omega^\alpha \} = \emptyset , \]
which implies that $\textup{CB}(K)<\alpha$. Furthermore, $\theta(\alpha,k)$ is homeomorphic to the topological sum of $k$ copies of $\theta(\alpha,1)$.

\section{The proof of the result}

For any $\alpha \in \omega_1$ and $k \in \omega \setminus \{ 0 \}$, let $S(\alpha,k)$ be the following statement:\vspace{1em}

\noindent\fbox{\begin{minipage}{\textwidth}
For any countable compact space $X$, any $K \subseteq X$ homeomorphic to $\theta(\alpha,k)$, any $a,b \in \R$ with $a<b$, and any $\delta >0$, there exists a bounded linear operator $\Phi \colon C(X) \to C([a,b])$ such that for any $f \in C(X)$:
\begin{enumerate}
    \item if $f$ is not the constant zero function, then $\Phi(f)$ is Besicovitch;
    \item $\lVert f\restriction_K \rVert \leq \lVert \Phi(f) \rVert \leq \lVert f \rVert$;
    \item $\displaystyle \min_{x \in K}f(x)-\delta \lVert f \rVert \leq \Phi(f)(t) \leq \max_{x \in K}f(x)+\delta \lVert f \rVert$ for all $t \in [a,b]$.
\end{enumerate}
\end{minipage}}

\vspace{1em}We will use transfinite induction to show that $S(\alpha,k)$ holds for all $\alpha \in \omega_1$ and $k \in \omega \setminus \{ 0 \}$. Let us note that the sole purpose of property (3) in the statement is to make the induction argument work.

\begin{lemma}\label{lemma}
    For any separable normed linear space $E$, there exists a continuous linear operator $T \colon E \to C([0,1])$ with $\lVert T \rVert \leq 1$ such that for every nonzero $u \in E$, the function $T(u)$ is Besicovitch and satisfies $T(u)(0)=T(u)(1)=0$.
\end{lemma}
\begin{proof}
    This lemma builds directly on \cite{JB15,JB26}. Let us therefore recall the result from these references that we utilize below. Let $\mathcal{B}$ be the space of all bounded functions $f \colon (0,1] \to \R$ whose restriction to $\big( \frac{1}{n+1} , \frac{1}{n} \big]$ is constant for every $n \in \N$. Equipped with the supremum norm, $\mathcal{B}$ becomes a (non-separable) Banach space that is linearly isometric to $\ell^\infty$. It was shown in \cite{JB15,JB26} that there exists a bounded linear operator $T_z \colon \mathcal{B} \to C([0,1])$ such that for any nonzero $f \in \mathcal{B}$, the function $T_z(f)$ is Besicovitch and satisfies $T_z(f)(0)=T_z(f)(1)=0$.

    Since $E$ is separable, there is a linear isometric embedding $\psi$ of $E$ into $\mathcal B$. Define the operator $T \colon E \to C([0,1])$ by
    \[ T(u)= \frac{1}{\lVert T_z \rVert} T_z \big( \psi (u) \big). \]
    Clearly, $T$ has the desired properties.
\end{proof}

\begin{proposition}\label{Prop1}
    The statement $S(0,1)$ holds.
\end{proposition}
\begin{proof}
    Let $X,K,a,b$ and $\delta$ satisfying the assumptions of $S(0,1)$ be given. We can assume without any loss of generality that $[a,b]=[0,1]$ and that $\delta<1$. Since $K$ is homeomorphic to the ordinal $\theta(0,1)=1$, the set $K$ is a singleton, say $K= \{ p \}$. Define an operator $\Phi_0 \colon C(X) \to C([0,1])$ by
    \[ \Phi_0(f)(t)= \frac{2-\delta t}{2} f(p). \]
    In other words, for any $f \in C(X)$, $\Phi_0(f)$ is the affine function on $[0,1]$ satisfying $\Phi_0(f)(0)=f(p)$ and $\Phi_0(f)(1)=(1-\delta/2)f(p)$. The final operator $\Phi$ is obtained by adding a perturbation operator \(R\), to be defined below, to $\Phi_0$.

    Let $T \colon C(X) \to C([0,1])$ be the operator from Lemma \ref{lemma}. For each $n \in \N$, let $u_n \colon \R \to \R$ be the affine function which satisfies $u_n(2^{-n})=0$, $u_n(2^{1-n})=1$. Formally, $u_n(t)=2^nt-1$ for any $t \in \R$. Define an operator $R \colon C(X) \to C([0,1])$ by
    \begin{equation*}
        R(f)(t)=
        \begin{cases}
            0 &\textup{if } t=0,\\
            2^{-n-1}\delta \, T(f)\big( u_n(t) \big) &\textup{if } t\in(2^{-n},2^{1-n}], n \in \N .
        \end{cases}
    \end{equation*}
    Since $T(f)$ is continuous and vanishes at $0$ and $1$, the continuity of $R(f)$ follows.
    
    Let $\Phi := \Phi_0+R$. It is clear that $\Phi_0$, $R$ and $\Phi$ are bounded linear operators. For any $f \in C(X)$, we have $\Phi(f)(0)=\Phi_0(f)(0)=f(p)$, which shows that $\lVert \Phi(f) \rVert \geq \lVert f \restriction_K \rVert$. On the other hand, for any $n \in \N$ and $t \in (2^{-n},2^{1-n}]$,
    \[ |\Phi(f)(t)| \leq |\Phi_0(f)(t)|+|R(f)(t)| \leq \frac{2-2^{-n} \delta}{2}\lVert f \rVert +2^{-n-1} \delta \lVert f \rVert = \lVert f \rVert , \]
    hence $\lVert \Phi(f) \rVert \leq \lVert f \rVert$. Moreover, again for any $n \in \N$ and $t \in (2^{-n},2^{1-n}]$,
    \begin{align*}
        \Phi(f)(t)&= f(p)-\frac{\delta t}{2} f(p)+R(f)(t) \leq f(p)+\frac{\delta t}{2}|f(p)|+2^{-n-1} \delta \lVert f \rVert\\
    &\leq f(p)+2^{-n}\delta \lVert f \rVert+2^{-n-1} \delta \lVert f \rVert \leq f(p)+\delta \lVert f \rVert =\max_{x \in K}f(x)+\delta \lVert f \rVert
    \end{align*}
    and
    \begin{align*}
        \Phi(f)(t)&= f(p)-\frac{\delta t}{2} f(p)+R(f)(t) \geq f(p)-\frac{\delta t}{2}|f(p)|-2^{-n-1} \delta \lVert f \rVert\\
    &\geq f(p)-2^{-n}\delta \lVert f \rVert-2^{-n-1} \delta \lVert f \rVert \geq f(p)-\delta \lVert f \rVert = \min_{x \in K}f(x)-\delta \lVert f \rVert .
    \end{align*}
    It remains to show that $\Phi(f)$ is a Besicovitch function (assuming $f$ is nonzero). Since $\Phi_0(f)$ is affine (and thus differentiable), it suffices to show that $R(f)$ is Besicovitch. Clearly, by the definition of $R(f)$ and the fact that $T(f)$ is Besicovitch, $R(f)$ does not have a one-sided derivative at any point $t \in (0,1]$. To prove that the right derivative of $R(f)$ at $0$ does not exist, note that for every $n \in \N$,
    \[ R(f)(2^{-n})=0=R(f)(0). \]
    This shows that if the right derivative of $R(f)$ at $0$ existed, it would be equal to $0$. Since $T(f)$ is a Besicovitch function, there exists $r \in (0,1)$ with $T(f)(r) \neq 0$. For each $n \in \N$, let $s_n \in (2^{-n},2^{1-n})$ be the point satisfying $u_n(s_n)=r$. Then
    \[ \Bigg| \frac{R(f)(s_n)-R(f)(0)}{s_n-0} \Bigg| = \Bigg| \frac{R(f)(s_n)}{s_n} \Bigg| = \frac{\big| 2^{-n-1}\delta \, T(f)(r) \big|}{s_n} \geq \frac{2^{-n-1}\delta}{2^{1-n}} |T(f)(r)| = \frac{\delta}{4} |T(f)(r)| \]
    for each $n \in \N$. Therefore, the right derivative of $R(f)$ at $0$ cannot be $0$.
\end{proof}

\begin{proposition}\label{Prop2}
    Let $\alpha \in \omega_1$. If $S(\alpha,1)$ holds, then so does $S(\alpha,k)$ for every $k \in \omega \setminus \{ 0 \}$.
\end{proposition}
\begin{proof}
    Assuming $S(\alpha,1)$ holds, we will proceed by induction. Given any $k \in \omega \setminus \{ 0 \}$ such that $S(\alpha,k)$ holds, let us show that $S(\alpha ,k+1)$ holds. Let $X$, $K$, $a,b$ and $\delta$ satisfying the assumptions of $S(\alpha,k+1)$ be given. We can assume without loss of generality that $[a,b]=[-1,3]$ and that $\delta<1$. Since $K$ is homeomorphic to $\theta (\alpha ,k+1)$, it can be expressed as $K=K_1 \cup K_2$, where $K_1,K_2$ are disjoint, $K_1$ is homeomorphic to $\theta (\alpha ,k)$ and $K_2$ is homeomorphic to $\theta (\alpha ,1)$. Using $S(\alpha ,k)$ for $X$, $K_1$, $a_1=-1$, $b_1=0$ and $\delta_1=\delta /2$, we obtain an operator $\Phi_1 \colon C(X) \to C([-1,0])$. In addition, we use $S(\alpha,1)$ with $X$, $K_2$, $a_2=2$, $b_2=3$ and $\delta_2=\delta /2$ to obtain an operator $\Phi_2 \colon C(X) \to C([2,3])$. Fix an arbitrary point $p \in K$ and define an operator $\Phi_0 \colon C(X) \to C([-1,3])$ as follows. For any $f \in C(X)$, let $\Phi_0 (f)$ be the unique function on $[-1,3]$ such that:
    \begin{enumerate}[label=\textup{(\roman*)}]
        \item $\Phi_0(f)(t)=\Phi_1(f)(t)$ for $t \in [-1,0]$;
        \item $\Phi_0(f)(t)=\Phi_2(f)(t)$ for $t \in [2,3]$;
        \item $\Phi_0(f)(1)=\big( 1-\frac{1}{2}\delta \big) f(p)$;
        \item $\Phi_0(f)$ is affine on each of the intervals $[0,1]$, $[1,2]$.
    \end{enumerate}
    The explicit formula for $\Phi_0(f)$ is
    \begin{equation*}
        \Phi_0(f)(t)=\begin{cases}
            \Phi_1(f)(t) &\textup{if } t \in [-1,0]\\            
            t\big( 1-\tfrac{1}{2}\delta \big)f(p)+(1-t)\Phi_1(f)(0) &\textup{if } t \in (0,1]\\
            (2-t)\big( 1-\tfrac{1}{2}\delta \big)f(p)+(t-1)\Phi_2(f)(2) &\textup{if } t \in (1,2)\\
            \Phi_2(f)(t) &\textup{if } t \in [2,3].
        \end{cases}
    \end{equation*}

    Let $T \colon C(X) \to C([0,1])$ be the operator from Lemma \ref{lemma}. For each $n \in \N$, let $u_n \colon \R \to \R$ and $v_n \colon \R \to \R$ be the increasing affine functions such that $u_n$ maps $[2^{-n},2^{1-n}]$ onto $[0,1]$ and $v_n$ maps $[2-2^{1-n},2-2^{-n}]$ onto $[0,1]$. Define an operator $R \colon C(X) \to C([-1,3])$ by
    \begin{equation*}
        R(f)(t)=
        \begin{cases}
            0 &\textup{if } t \in [-1,0] \cup [2,3] \\
            2^{-n-1}\delta \, T(f)\big( u_n(t) \big) &\textup{if } t\in(2^{-n},2^{1-n}], n \in \N\\
            2^{-n-1}\delta \, T(f)\big( v_n(t) \big) &\textup{if } t\in(2-2^{1-n},2-2^{-n}], n \in \N 
        \end{cases}
    \end{equation*}
    and let $\Phi := \Phi_0 + R$. Clearly, $\Phi$ is linear. For any $f \in C(X)$, we have
    \[ \lVert \Phi(f) \rVert \geq \lVert \Phi(f) \restriction_{[-1,0]} \rVert = \lVert \Phi_1(f) \rVert \geq \lVert f \restriction_{K_1} \rVert \]
    and
    \[ \lVert \Phi(f) \rVert \geq \lVert \Phi(f) \restriction_{[2,3]} \rVert = \lVert \Phi_2(f) \rVert \geq \lVert f \restriction_{K_2} \rVert .\]
    Thus, $\lVert \Phi(f) \rVert \geq \lVert f \restriction_{K} \rVert$. To show that $\lVert \Phi (f) \rVert \leq \lVert f \rVert$, first note that $\lVert \Phi(f) \restriction_{[-1,0]} \rVert = \lVert \Phi_1(f) \rVert \leq \lVert f \rVert$ and $\lVert \Phi(f) \restriction_{[2,3]} \rVert = \lVert \Phi_2(f) \rVert \leq \lVert f \rVert$. Also note that for any $n \in \N$ and $t \in (2^{-n},2^{1-n}]$,
    \begin{align*}
        |\Phi_0(f)(t)| &= \big| t\big( 1-\tfrac{1}{2}\delta \big)f(p)+(1-t)\Phi_1(f)(0) \big|\\
        &\leq t\big( 1-\tfrac{1}{2}\delta \big) \lVert f \rVert + (1-t)\lVert f \rVert = \big( 1-\tfrac12 \delta t \big)\lVert f \rVert \\
        &\leq \big( 1-\tfrac12 \delta \cdot 2^{-n} \big)\lVert f \rVert = (1-2^{-n-1} \delta) \lVert f \rVert ,
    \end{align*}
    and thus
    \[ |\Phi(f)(t)| \leq |\Phi_0(f)(t)|+|R(f)(t)| \leq (1-2^{-n-1} \delta) \lVert f \rVert+2^{-n-1} \delta \lVert f \rVert = \lVert f \rVert .\]
    Similarly one can show that $|\Phi(f)(t)| \leq \lVert f \rVert$ when $t \in (2-2^{1-n},2-2^{-n}]$ for some $n \in \N$.

    With $f \in C(X)$ fixed, denote $\mu(L):=\min \{ f(x); \, x \in L \}$, $M(L):=\max \{ f(x); \, x \in L \}$ for any nonempty compact set $L \subseteq K$. Let us prove that $\mu(K)-\delta \lVert f \rVert \leq \Phi(f)(t) \leq M(K)+\delta \lVert f \rVert$ for every $t \in [-1,3]$. When $t \in [-1,0] \cup [2,3]$, then $\Phi(f)(t)=\Phi_j(f)(t)$ for some $j \in \{ 1,2 \}$, and thus
    \[ \mu(K)-\delta \lVert f \rVert \leq \mu(K_j)-\delta_j \lVert f \rVert \leq \Phi(f)(t) \leq M(K_j)+\delta_j \lVert f \rVert \leq M(K)+\delta \lVert f \rVert . \]
    Let $t \in (0,2)$. Since $\Phi_0(f)\restriction_{[0,1]}$, $\Phi_0(f)\restriction_{[1,2]}$ are affine and
    \begin{align*}
        \Phi_0(f)(0)&=\Phi_1(f)(0)\leq M(K_1)+\delta_1 \lVert f \rVert \leq M(K)+\tfrac12 \delta \lVert f \rVert ,\\
        \Phi_0(f)(1)&=\big( 1-\tfrac{1}{2}\delta \big) f(p) = f(p)-\tfrac{1}{2}\delta f(p) \leq M(K)+\tfrac12 \delta \lVert f \rVert ,\\
        \Phi_0(f)(2)&=\Phi_2(f)(2)\leq M(K_2)+\delta_2 \lVert f \rVert \leq M(K)+\tfrac12 \delta \lVert f \rVert ,
    \end{align*}
    we have $\Phi_0(f)(t) \leq M(K)+\frac12 \delta \lVert f \rVert$, and hence
    \[ \Phi(f)(t)=\Phi_0(f)(t)+R(f)(t) \leq M(K)+\tfrac12 \delta \lVert f \rVert+ \tfrac14 \delta \lVert f \rVert \leq M(K)+\delta \lVert f \rVert . \]
    Similarly one can show that $\Phi(f)(t) \geq \mu(K)-\delta \lVert f \rVert$.

    Assuming $f$ is not the constant zero function, it remains to show that $\Phi(f)$ is Besicovitch. Since $\Phi_0(f) \restriction_{[0,1]}$, $\Phi_0(f) \restriction_{[1,2]}$, $R(f) \restriction_{[-1,0]}$, $R(f) \restriction_{[2,3]}$ are affine and $\Phi_0(f) \restriction_{[-1,0]}$, $\Phi_0(f) \restriction_{[2,3]}$ are Besicovitch, it suffices to show that $R(f) \restriction_{[0,2]}$ is Besicovitch. However, thanks to $T(f)$ being Besicovitch, it actually only remains to show that the right derivative of $R(f)$ at $0$ and the left derivative of $R(f)$ at $2$ do not exist. That can be easily achieved by mimicking the corresponding part of the proof of Proposition \ref{Prop1}.
\end{proof}

\begin{proposition}
    For any $\alpha \in \omega_1$, the statement $S(\alpha,1)$ holds.
\end{proposition}
\begin{proof}
    By Proposition \ref{Prop1}, the statement $S(0,1)$ holds. We proceed by transfinite induction. Given $\alpha \in \omega_1 \setminus \{ 0 \}$, assume $S(\beta,1)$ holds for all $\beta < \alpha$. Then, by Proposition \ref{Prop2}, $S(\beta,k)$ holds for all $\beta < \alpha$ and $k \in \omega \setminus \{ 0 \}$. Let $X,K,a,b$ and $\delta$ satisfying the assumptions of $S(\alpha,1)$ be given. We can assume without loss of generality that $[a,b]=[0,8]$ and that $\delta<1$.
    \begin{claim}
        There exists a point $p \in K$ and a sequence $(K_m)_{m=1}^\infty$ of nonempty compact subsets of $K$ such that
        \begin{enumerate}[label=\textup{(\arabic*)}]
        \item $\bigcup \{ K_m ; \, m \in \N \}=K \setminus \{ p \}$;
        \item $\textup{CB}(K_m)<\alpha$ for every $m \in \N$;
        \item for any neighbourhood $U$ of $p$, we have $K_m \subseteq U$ for all but finitely many $m \in \N$.
    \end{enumerate}
    \end{claim}
    \begin{claimproof}
        Since $K$ is homeomorphic to $\theta(\alpha,1)=\omega^\alpha+1$, there is a homeomorphism $\psi \colon \omega^\alpha+1 \to K$. Let $(\gamma_m)_{m=1}^\infty$ be an increasing sequence of ordinals converging to $\omega^\alpha$, with $\gamma_1=0$. For each $m \in \N$, let $L_m$ be the set of all ordinals $\gamma$ such that $\gamma_m \leq \gamma \leq \gamma_{m+1}$. Then $(L_m)_{m=1}^\infty$ is a sequence of nonempty compact subsets of $\omega^\alpha+1$ not containing the point $\omega^\alpha$. Hence, $\textup{CB}(L_m)<\alpha$ for all $m \in \N$ (recall the last paragraph of Section \ref{SectionPreliminaries}). Let $p:=\psi(\omega^\alpha)$ and $K_m := \psi (L_m)$ for $m \in \N$. Condition (1) is satisfied and, since $\bigcup \{ L_m ; \, m \in \N \} = (\omega^\alpha+1) \setminus \{ \omega^\alpha \}$, so is (2). Given a neighbourhood $U$ of $p$, the set $\psi^{-1}(U)$ is a neighbourhood of $\omega^\alpha$ in the space $\omega^\alpha+1$. Hence, as $(\gamma_m)_{m=1}^\infty$ converges to $\omega^\alpha$, we have $L_m \subseteq \psi^{-1}(U)$ for all but finitely many $m \in \N$. Condition (3) follows.
        \claimend
    \end{claimproof}
    For every $m \in \N$, let $a_m :=12 \cdot 2^{-m}$, $b_m :=16 \cdot 2^{-m}$ and $\delta_m :=2^{-m}\delta$. Using $S(\beta_m ,k_m)$ for $X$, $K_m$, $a_m$, $b_m$, $\delta_m$, we obtain an operator $\Phi_m \colon C(X) \to C([a_m,b_m])$. Let us define an operator $\Phi_0 \colon C(X) \to C([0,8])$ as follows. For any $f \in C(X)$, let $\Phi_0(f)$ be the unique function on $[0,8]$ such that for all $m \in \N$:
    \begin{enumerate}[label=\textup{(\roman*)}]
        \item $\Phi_0(f)(t)=\Phi_m(f)(t)$ for any $t \in [a_m,b_m]$;
        \item $\Phi_0(f)(10 \cdot 2^{-m})=\Phi_0(f)(0)=f(p)$;
        \item $\Phi_0(f)(9 \cdot 2^{-m})=\Phi_0(f)(11 \cdot 2^{-m})=(1-2^{-m}\delta)f(p)$;
        \item $\Phi_0(f)$ is affine on $[j \cdot 2^{-m},(j+1)2^{-m}]$ for each $j \in \{ 8,9,10,11 \}$.
    \end{enumerate}
    The explicit formula for $\Phi_0(f)(t)$, when $t \in [8 \cdot 2^{-m},12 \cdot 2^{-m}]=[b_{m+1},a_m]$, is
    \begin{equation*}
        \Phi_0(f)(t)=\begin{cases}
            (2^m t-8)(1-2^{-m}\delta)f(p)+(9-2^m t)\Phi_{m+1}(f)(b_{m+1}), &t\in [8 \cdot 2^{-m},9 \cdot 2^{-m})\\
            (1+\delta t-10 \cdot 2^{-m}\delta)f(p), &t\in [9 \cdot 2^{-m},10 \cdot 2^{-m})\\
            (1+10 \cdot 2^{-m}\delta-\delta t)f(p), &t\in [10 \cdot 2^{-m},11 \cdot 2^{-m})\\
            (2^mt-11)\Phi_m(f)(a_m)+(12-2^m t)(1-2^{-m}\delta)f(p), &t\in [11 \cdot 2^{-m},12 \cdot 2^{-m}].
        \end{cases}
    \end{equation*}

    \begin{figure}[htbp]
\centering
\begin{tikzpicture}[x=0.9cm, y=3.5cm, >=stealth] 

    \newcommand{\fp}{1.4}
    \newcommand{\fPDelta}{1.1}
    
    \draw[->, thick] (-0.5, 0) -- (16.5, 0) node[right] {$t$};
    \draw[->, thick] (0, -0.1) -- (0, 1.8) node[above] {$\Phi_0(f)(t)$};

    \node[right, font=\small\bfseries] at (6.0, 1.9) {Part of the graph of $\Phi_0(f)$};

    \draw[dashed, gray] (0, \fp) -- (16.5, \fp);
    \node[above right, black, font=\small] at (0, \fp) {$f(p)$};
    
    \draw[dashed, gray] (0, \fPDelta) -- (16.5, \fPDelta);
    \node[above right, black, font=\small] at (0, \fPDelta) {$(1-2^{-m}\delta)f(p)$};

    \node[below left, font=\small] at (0, 0) {$0$};

    \foreach \x/\label in {
        6/{$a_{m+1}$},
        8/{$b_{m+1}$}, 
        9/{$\frac{9}{2^m}$}, 
        10/{$\frac{10}{2^m}$}, 
        11/{$\frac{11}{2^m}$}, 
        12/{$a_m$},
        16/{$b_m$}
    } {
        \draw[dotted, gray] (\x, 0) -- (\x, 1.6);
        \node[below, font=\small] at (\x, 0) {\label};
    }

    \draw[thick, blue, samples=250, domain=6.0:8.0] plot (
        \x, 
        { 0.5 * (0.95 - 0.07*(\x-8.0) + 0.16*sin(90*(\x-6.0)) - 0.28*sin(180*(\x-6.0)) + 0.04*sin(1350*(\x-6.0)) + 0.025*sin(3150*(\x-6.0))) + 0.5*\fPDelta - 0.07*(\x-8.0) }
    );
    
    \node[blue, above] at (7.0, 1.5) {\small \(\Phi_{m+1}(f)\)};

    \draw[thick, red] (8.0, 1.025) -- (9.0, \fPDelta);
    \draw[thick, red] (9.0, \fPDelta) -- (10.0, \fp);
    \draw[thick, red] (10.0, \fp) -- (11.0, \fPDelta);
    \draw[thick, red] (11.0, \fPDelta) -- (12.0, 1.2); 

    \fill[red] (8.0, 1.025) circle (1.3pt);
    \fill[red] (9.0, \fPDelta) circle (1.3pt);
    \fill[red] (10.0, \fp) circle (1.3pt);
    \fill[red] (11.0, \fPDelta) circle (1.3pt);
    \fill[red] (12.0, 1.2) circle (1.3pt);

    \draw[thick, blue, samples=500, domain=12.0:16.0] plot (
        \x, 
        { 1.2 - 0.125*(\x-12) + 0.35*((\x-12.0)/4.0)^8 + 0.35*sin(45*(\x-12)) + 0.15*sin(135*(\x-12)) + 0.04*sin(900*(\x-12)) + 0.025*sin(2250*(\x-12)) 
          + 0.04 * (0.5 - abs(mod(15*(\x-12), 1) - 0.5)) - 0.04*(\x-12.0) }
    );
        
    \node[blue, above] at (14.0, 1.6) {\small \(\Phi_m(f)\)};
\end{tikzpicture}
\end{figure}
    
    \begin{claim}
        $\Phi_0(f)$ is a continuous function for any $f \in C(X)$.
    \end{claim}
    \begin{claimproof}
        Let $f \in C(X)$. It is clear that $\Phi_0(f)$ is continuous on $(0,8]$, so it suffices to show that $\Phi_0(f)$ is continuous at $0$ from the right. Let $d_m:=\max \{ |f(x)-f(p)| \, ; \, x \in K_m \}$ for each $m \in \N$. Thanks to (3), the continuity of $f$ implies that $(d_m)_{m=1}^\infty$ converges to $0$. For any $m \in \N$ and $t \in [a_m,b_m]$, the value $\Phi_0(f)(t)=\Phi_m(f)(t)$ satisfies
        \[ \min_{x \in K_m} f(x)-\delta_m \lVert f \rVert \leq \Phi_0(f)(t) \leq \max_{x \in K_m} f(x)+\delta_m \lVert f \rVert , \]
        and thus
        \[ f(p)-d_m-2^{-m}\delta \lVert f \rVert \leq \Phi_0(f)(t) \leq f(p)+d_m+2^{-m}\delta \lVert f \rVert . \]
        Consequently, denoting $\displaystyle A:= \bigcup_{m \in \N}[a_m,b_m]$, we obtain \[ \lim_{\begin{smallmatrix} t \to 0, \\ t\in A \end{smallmatrix}} \Phi_0(f)(t)=f(p). \]
        By (ii), (iii) and (iv), it follows that $\displaystyle \lim_{t \to 0^+} \Phi_0(f)(t)=f(p)=\Phi_0(f)(0)$. \claimend
    \end{claimproof}

    For all $m,n \in \N$, let
    \begin{align*}
         c_{m,n}&:=2^{-m}(8+2^{1-n}), & d_{m,n}&:=2^{-m}(10-2^{1-n}),\\
         c_{m,n}^*&:=2^{-m}(10+2^{1-n}), & d_{m,n}^*&:=2^{-m}(12-2^{1-n}).
    \end{align*}
    Note that $c_{m,1}=d_{m,1}=9 \cdot 2^{-m}$ and $c_{m,1}^*=d_{m,1}^*=11 \cdot 2^{-m}$. The sequences $(c_{m,n})_{n=1}^\infty$, $(c_{m,n}^*)_{n=1}^\infty$ are decreasing and converge to $8 \cdot 2^{-m}$ and $10 \cdot 2^{-m}$, respectively, while $(d_{m,n})_{n=1}^\infty$, $(d_{m,n}^*)_{n=1}^\infty$ are increasing and converge to $10 \cdot 2^{-m}$ and $12 \cdot 2^{-m}$, respectively.

    
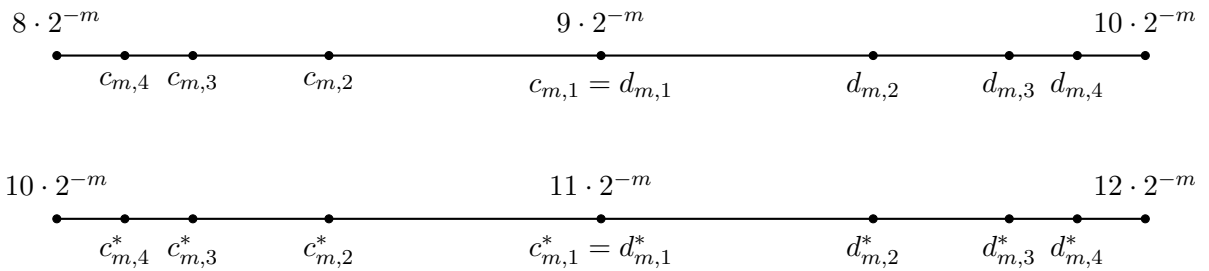
\begin{figure}[htbp]
    \centering
    \begin{tikzpicture}[x=7.2cm, y=1.2cm]
    \begin{scope}[shift={(0,0.9)}]
        \draw[thick] (8,0) -- (10,0);        
        \filldraw (8,0) circle (1.5pt);
        \filldraw (8.125,0) circle (1.5pt);
        \filldraw (8.25,0) circle (1.5pt);
        \filldraw (8.5,0) circle (1.5pt);
        \filldraw (9,0) circle (1.5pt);
        \filldraw (9.5,0) circle (1.5pt);
        \filldraw (9.75,0) circle (1.5pt);
        \filldraw (9.875,0) circle (1.5pt);
        \filldraw (10,0) circle (1.5pt);
        \node[above=5pt] at (8,0) {$8 \cdot 2^{-m}$};
        \node[above=5pt] at (9,0) {$9 \cdot 2^{-m}$};
        \node[above=5pt] at (10,0) {$10 \cdot 2^{-m}$};
        
        \begin{scope}[every node/.style={below, inner sep=1pt}]
            \node[yshift=-5pt] at (8.125,0) {$c_{m,4}$};
            \node[yshift=-5pt] at (8.25,0) {$c_{m,3}$};
            \node[yshift=-5pt] at (8.5,0) {$c_{m,2}$};
            \node[yshift=-5pt] at (9,0) {$c_{m,1} = d_{m,1}$};
            \node[yshift=-5pt] at (9.5,0) {$d_{m,2}$};
            \node[yshift=-5pt] at (9.75,0) {$d_{m,3}$};
            \node[yshift=-5pt] at (9.875,0) {$d_{m,4}$};
        \end{scope}
    \end{scope}

    \begin{scope}[shift={(0,-0.9)}]
        \draw[thick] (8,0) -- (10,0);        
        \filldraw (8,0) circle (1.5pt);
        \filldraw (8.125,0) circle (1.5pt);
        \filldraw (8.25,0) circle (1.5pt);
        \filldraw (8.5,0) circle (1.5pt);
        \filldraw (9,0) circle (1.5pt);
        \filldraw (9.5,0) circle (1.5pt);
        \filldraw (9.75,0) circle (1.5pt);
        \filldraw (9.875,0) circle (1.5pt);
        \filldraw (10,0) circle (1.5pt);
        \node[above=5pt] at (8,0) {$10 \cdot 2^{-m}$};
        \node[above=5pt] at (9,0) {$11 \cdot 2^{-m}$};
        \node[above=5pt] at (10,0) {$12 \cdot 2^{-m}$};
        
        \begin{scope}[every node/.style={below, inner sep=1pt}]
            \node[yshift=-5pt] at (8.125,0) {$c_{m,4}^*$};
            \node[yshift=-5pt] at (8.25,0) {$c_{m,3}^*$};
            \node[yshift=-5pt] at (8.5,0) {$c_{m,2}^*$};
            \node[yshift=-5pt] at (9,0) {$c_{m,1}^* = d_{m,1}^*$};
            \node[yshift=-5pt] at (9.5,0) {$d_{m,2}^*$};
            \node[yshift=-5pt] at (9.75,0) {$d_{m,3}^*$};
            \node[yshift=-5pt] at (9.875,0) {$d_{m,4}^*$};
        \end{scope}
    \end{scope}
    \end{tikzpicture}
    \end{figure}
    
    For all $m,n \in \N$, let $u_{m,n}$ and $v_{m,n}$ be the increasing affine functions mapping $[c_{m,n+1},c_{m,n}]$ and $[d_{m,n},d_{m,n+1}]$ onto $[0,1]$, respectively. Similarly, let their starred counterparts $u_{m,n}^*$ and $v_{m,n}^*$ be the increasing affine functions mapping $[c_{m,n+1}^*,c_{m,n}^*]$ and $[d_{m,n}^*,d_{m,n+1}^*]$ onto $[0,1]$.
    
    Let $T \colon C(X) \to C([0,1])$ be the operator from Lemma \ref{lemma}. Define $R \colon C(X) \to C([0,8])$ by
    \begin{equation*}
        R(f)(t)=\begin{cases}
            0 &\textup{if } t=0,\\
            0 &\textup{if } t \in[a_m,b_m] \cup \{ 10 \cdot 2^{-m} \} \textup{ for some } m \in \N,\\
            2^{-m-n}\delta \, T(f) (u_{m,n}(t)) &\textup{if } t\in (c_{m,n+1},c_{m,n}] \textup{ for } m,n \in \N,\\
            2^{-m-n}\delta \, T(f) (v_{m,n}(t)) &\textup{if } t\in (d_{m,n},d_{m,n+1}] \textup{ for } m,n \in \N,\\
            2^{-m-n}\delta \, T(f) (u_{m,n}^*(t)) &\textup{if } t\in (c_{m,n+1}^*,c_{m,n}^*] \textup{ for } m,n \in \N,\\
            2^{-m-n}\delta \, T(f) (v_{m,n}^*(t)) &\textup{if } t\in (d_{m,n}^*,d_{m,n+1}^*] \textup{ for } m,n \in \N .
        \end{cases}
    \end{equation*}
    It is easy to see that $R(f)$ is a continuous function for any $f \in C(X)$, and that $\Phi_0$ and $R$ are linear maps. Let $\Phi:=\Phi_0+R$.
    \begin{claim}
        For any $f \in C(X)$, we have $\lVert f\restriction_K \rVert \leq \lVert \Phi(f) \rVert \leq \lVert f \rVert$.
    \end{claim}
    \begin{claimproof}
        Let $f \in C(X)$. For any $m \in \N$,
        \[ \lVert \Phi(f) \rVert \geq \lVert \Phi(f)\restriction_{[a_m,b_m]} \rVert = \lVert \Phi_m(f) \rVert \geq \lVert f\restriction_{K_m} \rVert , \]
        hence $\lVert f\restriction_K \rVert \leq \lVert \Phi(f) \rVert$. On the other hand, for any $m \in \N$ and $t \in [a_m,b_m]$,
        \[ |\Phi(f)(t)|=|\Phi_m(f)(t)| \leq \lVert \Phi_m(f) \rVert \leq \lVert f \rVert . \]
        When $t \in (c_{m,n+1},c_{m,n}]$ for some $m,n \in \N$, then
    \begin{align*}
        |\Phi(f)(t)| &\leq |\Phi_0(f)(t)|+|R(f)(t)|\\
        &\leq (2^m t-8)(1-2^{-m}\delta)\lVert f \rVert+(9-2^m t)\lVert f \rVert + 2^{-m-n}\delta \lVert f \rVert\\
        &= (2^m t-\delta t-8+8 \cdot 2^{-m}\delta+9-2^m t+ 2^{-m-n}\delta )\lVert f \rVert\\
        &= \lVert f \rVert+(8 \cdot 2^{-m}-t+ 2^{-m-n})\delta \lVert f \rVert \\
        &\leq \lVert f \rVert+(8 \cdot 2^{-m}-c_{m,n+1}+ 2^{-m-n})\delta \lVert f \rVert = \lVert f \rVert .
    \end{align*}
    This shows that $|\Phi(f)(t)| \leq \lVert f \rVert$ when $t \in (8 \cdot 2^{-m},9 \cdot 2^{-m}]$ for some $m \in \N$. Similarly one can deduce the same inequality when $t$ belongs to any of the intervals $(j \cdot 2^{-m},(j+1) 2^{-m})$, $j \in \{ 9,10,11 \}$. \claimend
    \end{claimproof}

    \begin{claim}
        For any $f \in C(X)$ and $t \in [0,8]$,
        \[ \min_{x \in K}f(x)-\delta \lVert f \rVert \leq \Phi(f)(t) \leq \max_{x \in K}f(x)+\delta \lVert f \rVert . \]
    \end{claim}
    \begin{claimproof}
        Let $f \in C(X)$ and denote $\mu (F):=\min \{ f(x); \, x \in F \}$, $M(F):=\max \{ f(x); \, x \in F \}$ for any nonempty compact set $F \subseteq K$. For any $m \in \N$ and $t \in [a_m,b_m]$, we have $\Phi(f)(t)=\Phi_m(f)(t)$, and thus
        \[ \mu(K)-\delta \lVert f \rVert \leq \mu(K_m)-\delta_m \lVert f \rVert \leq \Phi(f)(t) \leq M(K_m)+\delta_m \lVert f \rVert \leq M(K)+\delta \lVert f \rVert . \]
        Moreover, since
        \begin{align*}
            \Phi_0(f)(8 \cdot 2^{-m})&=\Phi_{m+1}(f)(b_{m+1})\leq M(K_{m+1})+\delta_{m+1}\lVert f \rVert \leq M(K)+\tfrac12 \delta \lVert f \rVert ,\\
            \Phi_0(f)(10 \cdot 2^{-m})&=f(p) \leq M(K),\\
            \Phi_0(f)(12 \cdot 2^{-m})&=\Phi_m(f)(a_m)\leq M(K_m)+\delta_m\lVert f \rVert \leq M(K)+\tfrac12 \delta \lVert f \rVert ,\\
            \Phi_0(f)(9 \cdot 2^{-m}) &= \Phi_0(f)(11 \cdot 2^{-m}) = (1-2^{-m}\delta)f(p) \leq (1-2^{-m}\delta) M(K) \leq M(K)+ \tfrac12 \delta \lVert f \rVert ,
        \end{align*}
        it follows from (iv) that for all $t \in [8 \cdot 2^{-m},12 \cdot 2^{-m}]$,
        \[ \Phi(f)(t)=\Phi_0(f)(t)+R(f)(t) \leq M(K)+ \tfrac12 \delta \lVert f \rVert+ 2^{-m} \delta \lVert f \rVert \leq M(K)+\delta \lVert f \rVert . \]
        Similarly it can be shown that $\Phi(f)(t) \geq \mu(K)-\delta \lVert f \rVert$ whenever $t \in [8 \cdot 2^{-m},12 \cdot 2^{-m}]$ for some $m \in \N$. \claimend
    \end{claimproof}
    Given nonzero $f \in C(X)$, let us show that $\Phi(f)$ is Besicovitch. For $t \in (0,8)$ which is not equal to $a_m$, $b_m$ or $10 \cdot 2^{-m}$ for any $m \in \N$, it is clear that $\Phi(f)$ does not have a one-sided derivative at $t$. Moreover, for any $m \in \N$, neither the left derivative of $\Phi(f)$ at $b_m$ nor the right derivative at $a_m$ exists. The reasoning shown in the last paragraph of the proof of Proposition \ref{Prop1} can easily be adjusted to work here as well. We conclude that for any $m \in \N$, $\Phi(f)$ does not have a one-sided derivative at any of the points $a_m$, $b_m$ and $10 \cdot 2^{-m}$. It remains to show that the right derivative of $\Phi(f)$ at $0$ does not exist. Since for each $m \in \N$
    \[ \Phi(f)(10 \cdot 2^{-m})=\Phi_0(f)(10 \cdot 2^{-m})=f(p)=\Phi(f)(0), \]
    we deduce that if the right derivative of $\Phi(f)$ at $0$ existed, it would be equal to $0$. On the other hand, for any $m \in \N$,
    \begin{align*}
        \frac{\Phi(f)(9 \cdot 2^{-m})-\Phi(f)(0)}{9 \cdot 2^{-m}-0}&= \frac{\Phi_0(f)(9 \cdot 2^{-m})+R(f)(c_{m,1})-f(p)}{9 \cdot 2^{-m}}\\
        &=\frac{(1-2^{-m}\delta)f(p)+0-f(p)}{9 \cdot 2^{-m}}\\
        &=\frac{-2^{-m}\delta f(p)}{9 \cdot 2^{-m}}=-\frac{\delta f(p)}{9}.
    \end{align*}
    This shows that if $f(p)\neq 0$, then the right derivative of $\Phi(f)$ at $0$ does not exist. Assume $f(p)=0$. Then for any $m \in \N$ and $t \in [9 \cdot 2^{-m},10 \cdot 2^{-m}]$, we have $\Phi_0(f)(t)=0$, and hence $\Phi(f)(t)=R(f)(t)$. Since $T(f)$ is a Besicovitch function on $[0,1]$, there is $r \in (0,1)$ such that $T(f)(r)\neq 0$. For each $m \in \N$, let $t_m \in (d_{m,1},d_{m,2})$ be the point with $v_{m,1}(t_m)=r$. Then
    \begin{align*}
        \bigg| \frac{\Phi(f)(t_m)-\Phi(f)(0)}{t_m-0} \bigg|&=\frac{|R(f)(t_m)-f(p)|}{t_m}=\frac{|2^{-m-1}\delta \, T(f)(r)-0|}{t_m}\\
        &=\frac{2^{-m-1}\delta \, |T(f)(r)|}{t_m} \geq \frac{2^{-m-1}\delta \, |T(f)(r)|}{10 \cdot 2^{-m}}=\frac{\delta}{20}|T(f)(r)|
    \end{align*}
    for any $m \in \N$, which shows that the right derivative of $\Phi(f)$ at $0$ does not exist.
\end{proof}

\begin{corollary}
    For any $\alpha \in \omega_1$ and $k \in \omega \setminus \{ 0 \}$, the statement $S(\alpha ,k)$ holds.
\end{corollary}

\begin{proof}[\textbf{Proof of Theorem \ref{MainResult}.}]
    Let $X$ be a countable compact space. The case when $X=\emptyset$ is trivial, so assume $X \neq \emptyset$. Then $X$ is homeomorphic to $\theta (\alpha,k)$ for some $\alpha \in \omega_1$ and $k \in \omega \setminus \{ 0 \}$. By $S(\alpha,k)$ used for $K=X$, $[a,b]=[0,1]$ and arbitrary $\delta >0$, there exists a bounded linear operator $\Phi \colon C(X) \to C([0,1])$ such that
    \begin{enumerate}[label=\textup{(\arabic*)}]
        \item $\Phi(f)$ is a Besicovitch function for any nonzero $f \in C(X)$;
        \item $\lVert f\restriction_K \rVert \leq \lVert \Phi(f) \rVert \leq \lVert f \rVert$ for any $f \in C(X)$.
    \end{enumerate}
    Since $K=X$, (2) shows that $\Phi$ is an isometric embedding of $C(X)$ into $C([0,1])$. Letting $Y$ be the image of $\Phi$, we obtain the desired subspace.
\end{proof}

\bibliographystyle{alpha}
\bibliography{citace}

\end{document}